\documentclass[final]{siamltex}

\usepackage{amsfonts}
\usepackage{amsmath}
\usepackage{amssymb}
\usepackage{color}
\usepackage{graphicx}
\usepackage{times}
\usepackage[round]{natbib}
\usepackage[T1]{fontenc}

\newcommand {\Vector} {\mbox {vec}}
\newcommand {\Matrix} {\mbox {mat}}

\newcommand {\col} {\mbox {\rm col}\hspace{0.05em}}
\newcommand {\DA} {\Delta A}

\newcommand {\Db} {\Delta b}
\newcommand {\Dd} {\Delta d}
\newcommand {\Ds} {\Delta s}
\newcommand {\Du} {\Delta u}
\newcommand {\Dv} {\Delta v}
\newcommand {\Dx} {\Delta x}
\newcommand {\Dy} {\Delta y}

\newcommand {\reals} {{\mathbb R}}

\newcommand {\sigmamin} {\sigma_{\min}}

\newcommand {\sigmamina} {\sigma_{\min}}
\newcommand {\strutwidth} {0pt}
\newcommand {\sub} [1] {_{\mbox {\scriptsize \rm #1}}}

\newcommand {\tr} {\mathop {\mbox {\rm tr}} \kern 0.1em}
\newcommand {\setarray} {\relax}
\newcommand {\onetwo} [2] {\left[ \setarray \begin {array} {c c} #1& #2 \end {array} \right]}

\newcommand {\scaleA} {\mathcal {A}}
\newcommand {\scaleB} {\mathcal {B}}
\newcommand {\scaleX} {\mathcal {X}}

\newcommand {\vds} {\mbox {\bfseries \scshape v}}
\newcommand {\boldtheta} {{\boldsymbol \theta}}
\newcommand {\boldkappa} {{\boldsymbol \kappa^{}_{\boldsymbol 2}}}
\newcommand {\boldkappasquared} {{\boldsymbol \kappa^{2}_{\boldsymbol 2}}}

\newcommand {\JxA} {J_x [\Vector (A)]}
\newcommand {\Jxb} {J_x(b)}

\newcommand {\X} {X}

\title{Condition Number of\\ Full Rank Linear least-squares Solutions}

\title{Spectral Condition Numbers for\\ Full Rank Linear least-squares Solutions}

\title{Bounds for Spectral Condition Numbers of\\ Full Rank Linear least-squares Solutions}

\title{Spectral Condition Numbers of\\ Full Rank Linear least-squares Solutions}

\title{Nuclear Norms of Rank 2 Matrices for Spectral Condition Numbers of Full Rank Linear Least Squares Solutions}

\author{Joseph F. Grcar\thanks{6059 Castlebrook Drive, Castro Valley, California 94552 USA (jfgrcar@comcast.net).}}

\begin{document}

\maketitle

\begin{abstract}
The condition number of solutions to full rank linear least-squares problem are shown to be given by an optimization problem that involves nuclear norms of rank 2 matrices. The condition number is with respect to the least-squares coefficient matrix and 2-norms. It depends on three quantities each of which can contribute  ill-conditioning. The literature presents several estimates for this condition number with varying results; even standard reference texts contain serious overestimates. The use of the nuclear norm affords a single derivation of the best known lower and upper bounds on the condition number and shows why there is unlikely to be a closed formula. 
\end{abstract}




\pagestyle{myheadings}
\thispagestyle{plain}
\markboth{JOSEPH F. GRCAR}{LINEAR LEAST SQUARES SOLUTION}

\section {Introduction}

\subsection {Purpose}

Linear least-squares problems, in the form of statistical regression analyses, are a basic tool of investigation in both the physical and the social sciences, and consequently they are an important computation. 

This paper develops a single methodology that determines tight lower and upper estimates of condition numbers for several problems involving linear least-squares. The condition numbers are with respect to the matrices in the problems and scaled $2$-norms. The problems are: orthogonal projections and least-squares residuals \citep {Grcar2010f}, minimum $2$-norm solutions of underdetermined equations \citep {Grcar2010b}, and in the present case, the solution of overdetermined equations
\begin {equation}
\label {eqn:LLS}
x = \arg \min_u \| b - A u \|_2 \quad \Rightarrow \quad A^t A x = A^t b \, ,
\end {equation}
where $A$ is an $m \times n$ matrix of full column rank (hence, $m \ge n$). Some presentations of error bounds contain formulas that can severely overestimate the condition number, including the SIAM documentation for the LAPACK software.

This introduction provides some background material. Section \ref {sec:definition} discusses definitions of condition numbers. Section \ref {sec:brief} describes the estimate and provides an example; this material is appropriate for presentation in class. Section \ref {sec:derivation} proves that the condition number varies from the estimate within a factor of $\sqrt 2$. The derivation relies on a formula for the nuclear norm of a matrix. (This norm is the sum of the singular values including multiplicities, and is also known as the trace norm.) Section \ref {sec:comparison} examines overestimates in the literature. Section \ref {sec:advanced} evaluates the nuclear norm of rank $2$ matrices (lemma \ref {lem:rank2}). 

\subsection {Prior Work}

Ever since \citet {Legendre1805-HM} and \citet {Gauss1809-HM} invented the method of least-squares, the problems had been solved by applying various forms of elimination to the normal equations, $A^t A x = A^t b$ in equation (\ref {eqn:LLS}). Instead, \citet {Golub1965} suggested applying Householder transformations directly to $A$, which removed the need to calculate $A^t A$. However, \citet [p.\ 144] {GolubWilkinson1966} reported that $A^tA$ was still ``relevant to some extent'' to the accuracy of the calculation because they found that $A^tA$ appears in a bound on perturbations to $x$ that are caused by perturbations to $A$. Their discovery was ``something of a shock'' \citep [p.\ 241] {vanderSluis1975}. 

The original error bound of \citet [p.\ 144, eqn.\ 43] {GolubWilkinson1966} was difficult to interpret because of an assumed scaling for the problem. \citet [pp.\ 15, 17, top] {Bjorck1967a} derived a bound by the augmented matrix approach that was suggested to him by Golub. \citet [pp.\ 224--226] {Wedin1973} re-derived the bound from his study of the matrix pseudoinverse and exhibited a perturbation to the matrix that attains the leading term. Van der Sluis (1975, p.\ 251, eqn.\ 5.8) also derived Bj\"orck's bound and introduced a simplification of the formula and a geometric interpretation of the leading term. \citet [p.\ 31, eqn.\ 1.4.28] {Bjorck1996} later followed Wedin in basing the derivation of his bound on the pseudoinverse. \citet [p.\ 1189, eqn.\ 2.4 and line --6] {Malyshev2003} derived a lower bound for the condition number thereby proving that his formula and the coefficient in Bj\"orck's bound are quantifiably tight estimates of the spectral condition number. In contrast, condition numbers with respect to Frobenius norms have exact formulas that have been given in various forms by \citet {Geurts1982}, \citet {Gratton1996}, and \citet {Malyshev2003}.

\section {Condition numbers}
\label {sec:definition}

\subsection {Error bounds and definitions of condition numbers}

The oldest way to derive perturbation bounds is by differential calculus. If $y = f(x)$ is a vector valued function of the vector $x$ whose partial derivatives are continuous, then the partial derivatives give the best estimate of the change to $y$ for a given change to $x$
\begin {equation}
\label {eqn:approximation-1}
\Dy = f(x + \Dx) - f(x) \approx J_f (x) \, \Dx
\end {equation}
where $J_f (x)$ is the Jacobian matrix of the partial derivatives of $y$ with respect to $x$. The magnitude of the error in the first order approximation (\ref {eqn:approximation-1}) is bounded by Landau's little $o ( \| \Dx \| )$ for all sufficiently small $\| \Dx \|$.\footnote {The continuity of the partial derivatives establishes the existence of the Fr\'echet derivative and its representation by the Jacobian matrix. The definition of the Fr\'echet derivative is responsible for the error in equation (\ref {eqn:approximation-1}) being $o ( \| \Dx \| )$. The order of the error terms is independent of the norm because all norms for finite dimensional spaces are equivalent \citep [p.\ 54, thm.\ 1.7] {Stewart1990}.} Thus $J_f (x) \, \Dx$ is the unique linear approximation to $\Dy$ in the vicinity of $x$.\footnote {Any other linear function added to $J_f (x) \, \Dx$ differs from $\Dy$ by ${\mathcal O} (\| \Dx \|)$ and therefore does not provide a $o ( \| \Dx \| )$ approximation.} Taking norms produces a perturbation bound,
\begin {equation}
\label {eqn:calculus-1}
\| \Dy \| \le \| J_f (x) \| \, \| \Dx \| + o (\| \Dx \|) \, .
\end {equation}
Equation (\ref {eqn:calculus-1}) is the smallest possible bound on $\| \Dy \|$ in terms of $\| \Dx \|$ provided the norm for the Jacobian matrix is induced from the norms for $\Dy$ and $\Dx$. In this case for each $x$ there is some $\Dx$, which is nonzero but may be chosen arbitrarily small, so the bound (\ref {eqn:calculus-1}) is attained to within the higher order term, $o (\| \Dx \|)$. There may be many  ways to define condition numbers, but because equation (\ref {eqn:calculus-1}) is the smallest possible bound, any definition of a condition number for use in bounds equivalent to (\ref {eqn:calculus-1}) must arrive at the same value, $\chi_y (x) = \| J_f (x) \|$.\footnote {A theory of condition numbers in terms of Jacobian matrices was developed by \citet [p.\ 292, thm.\ 4] {Rice1966}. Recent presentations of the formula $\chi_y (x) = \| J_f (x) \|$ are given by \citet [p.\ 44] {Chatelin1996}, \citet [p.\ 27] {Deuflhard2003}, \citet [p.\ 39] {Quarteroni2000}, and \citet [p.\ 90] {Trefethen1997}.} The matrix norm may be too complicated to have an explicit formula, but tight estimates can be derived as in this paper. 

\subsection {One or separate condition numbers}
\label {sec:separate}

Many problems depend on two parameters $u$, $v$ which may consist of the entries of a matrix and a vector (for example). In principle it is possible to treat the parameters altogether.\footnote {As will be seen in table \ref {tab:LLS}, \citet {Gratton1996} derived a joint condition number of the least-squares solution with respect to a Frobenius norm of the matrix and vector that define the problem.}  A condition number for $y$ with respect to joint changes in $u$ and $v$ requires a common norm for perturbations to both. Such a norm is 
\begin {equation}
\label {eqn:joint-norm}
\max \big\{ \| \Du \|, \, \| \Dv \| \big\} \, .
\end {equation}
A single condition number then follows that appears in an optimal error bound,
\begin {equation}
\label {eqn:single}
\| \Dy \| \le   \| J_f (u, v) \| \, \max \big\{ \| \Du \|, \, \| \Dv \| \big\} + o \left(\max \big\{ \| \Du \|, \, \| \Dv \| \big\} \right) .
\end {equation}
The value of the condition number is again $\chi_y(u, v) = \| J_f (u, v) \|$ where the matrix norm is induced from the norm for $\Dy$ and the norm in equation (\ref {eqn:joint-norm}). 

Because $u$ and $v$ may enter into the problem in much different ways, it is customary to treat each separately.  This approach recognizes that the Jacobian matrix is a block matrix
\begin {displaymath}
J_f (u, v) = \onetwo {J_{f_1} (u)} {J_{f_2} (v)}
\end {displaymath}
where the functions $f_1 (u) = f(u, v)$ and $f_2(v) = f(u, v)$ have $v$ and $u$ fixed, respectively. 
The first order differential approximation (\ref {eqn:approximation-1}) is unchanged but is rewritten with separate terms for $u$ and $v$,
\begin {equation}
\label {eqn:approximation-2}
\Dy \approx J_{f_1} (u) \, \Du + J_{f_2} (v) \, \Dv \, .
\end {equation}
Bound (\ref {eqn:single}) then can be weakened by applying norm inequalities,
\begin {eqnarray}
\nonumber
\| \Dy \|& \le& \| J_{f_1} (u)  \Du + J_{f_2} (v) \Dv \| + o \left(\max \big\{ \| \Du \|, \, \| \Dv \| \big\} \right)\\
\noalign {\smallskip}
\label {eqn:double}
& \le& \left( \strut \| J_{f_1} (u) \| + \| J_{f_2} (v) \| \right) \, \max \big\{ \| \Du \|, \, \| \Dv \| \big\}\\
\nonumber 
&& \hspace* {12em} {} + o \left(\max \big\{ \| \Du \|, \, \| \Dv \| \big\} \right) \, .
\end {eqnarray}
The coefficients $\chi_y(u) = \| J_{f_1} (u) \|$ and $\chi_y(v) = \| J_{f_2} (v) \|$ are the separate condition numbers of $y$ with respect to $u$ and $v$, respectively. 

These two different approaches lead to error bounds (\ref {eqn:single}, \ref {eqn:double}) that differ by at most a factor of $2$ because it can be shown \citep {Grcar2010f}
\begin {equation}
\label {eqn:sum-6}
{\chi_y (u) + \chi_y (v) \over 2} \le \chi_y (u, v) \le \chi_y (u) + \chi_y (v) \, .
\end {equation}
Thus, for the purpose of deriving tight estimates of joint condition numbers, it suffices to consider $\chi_y (u)$ and $\chi_y (v)$ separately. 

\section {Conditioning of the least-squares solution}
\label {sec:brief}

\subsection {Reason for considering matrices of full column rank}
\label {sec:reason}

The linear least-squares problem (\ref {eqn:LLS}) does not have an unique solution when $A$ does not have full column rank. A specific $x$ can be chosen such as the one of minimum norm. However, small changes to $A$ can still produce large changes to $x$.\footnote {If $A$ does not have full column rank, then for every nonzero vector $z$ in the right null space of the matrix, $(A + b z^t) (z / z^t z) = b$. Thus, a change to the matrix of norm $\| b \|_2 \, \|z \|_2$ changes the solution from $A^\dag b$ to $z /\| z \|_2^2$.} In other words, a condition number of $x$ with respect to rank deficient $A$ does not exist or is ``infinite.'' Perturbation bounds in the rank deficient case can be found by restricting changes to the matrix, for which see \citet [p.\ 30, eqn.\ 1.4.26] {Bjorck1996} and \citet [p.\ 157, eqn.\ 5.3] {Stewart1990}. That theory is beyond the scope of the present discussion.

\subsection {The condition numbers}
\label {sec:brief-1}

This section summarizes the results and presents an example. Proofs are in section \ref {sec:derivation}. It is assumed that $A$ has full column rank and the solution $x$ of the least-squares problem (\ref {eqn:LLS}) is not zero. The solution is proved to have a condition number $\chi_x(A)$ with respect to $A$ within the limits,
\begin {equation}
\label {eqn:chixA}
\fbox {$\boldkappa \, \sqrt { \strut [\vds  \tan (\boldtheta)]^2 + 1} $} \; \le \; \chi_x (A) \; \le \; \fbox {$\boldkappa [\vds \tan (\boldtheta) + 1]$} \, ,
\end {equation}
where $\boldkappa$, $\vds$, and $\boldtheta$ are defined below; they are bold to emphasize they are the values in the tight estimates of the condition number. There is also condition number with respect to $b$,
\begin {equation}
\label {eqn:chixb}
\chi_x (b) = \fbox {$\vds \sec (\boldtheta)$} \, .
\end {equation}
These condition numbers $\chi_x(A)$ and $\chi_x (b)$ are for measuring the perturbations to $A$, $b$, and $x$ by the following scaled $2$-norms,
\begin {equation}
\label {eqn:specific-scaled-norms}
{\| \DA \|_2 \over \| A \|_2} \, ,
\qquad
{\| \Db \|_2 \over \| b \|_2} \, ,
\qquad
{\| \Dx \|_2 \over \| x \|_2} \, .
\end {equation}
Like equation (\ref {eqn:double}), the two condition numbers appear in error bounds of the form,\footnote {The constant denominators $\| A \|_2$ and $\| b \|_2$ could be discarded from the $o$ terms because only the order of magnitude of the terms is pertinent.} 
\begin {equation}
\label {eqn:error-bound}
{\| \Dx \|_2 \over \| x \|_2} \le \chi_x (A) {\| \DA \|_2 \over \| A \|_2} + \chi_x (b) {\| \Db \|_2 \over \| b \|_2} + o \left( \max \left\{ {\| \DA \|_2 \over \| A \|_2}, \, {\| \Db \|_2 \over \| b \|_2} \right\} \right) ,
\end {equation}
where $x + \Dx$ is the solution of the perturbed problem,
\begin {equation}
\label {eqn:perturbed-problem}
x + \Dx = \arg \min_u \| (b + \Db) - (A + \DA) u \|_2 \, .
\end {equation}

The quantities $\boldkappa$, $\vds$, and $\boldtheta$ in the formulas (\ref {eqn:chixA}, \ref {eqn:chixb}) are
\begin {equation}
\label {eqn:three}
\boldkappa = {\| A \|_2 \over \sigmamina} \qquad 
\vds = {\| A x \|_2 \over \| x \|_2 \, \sigmamina} \qquad
\tan (\boldtheta) = {\| r \|_2 \over \| A x \|_2}
\end {equation}
where $\boldkappa$ is the spectral matrix condition number of $A$ ($\sigmamin$ is the smallest singular value of $A$), $\vds$ is van der Sluis's ratio between $1$ and $\boldkappa$,\footnote {The formulas of \citet [p.\ 251] {vanderSluis1975} contain in his notation $R(x_0) / \sigma_n$, which is the present $\vds$.} $\boldtheta$ is the angle between $b$ and $\col (A)$,\footnote {The notation $\col (A)$ is the column space of $A$.} and $r = b - Ax$ is the least-squares residual. 
\smallskip
\begin {enumerate}
\item $\boldkappa$ depends only on the extreme singular values of $A$.
\item $\boldtheta$ depends only on the ``angle of attack'' of $b$ with respect to $\col (A)$.
\item If $A$ is fixed, then $\vds$ depends on the orientation of $b$ to $\col (A)$ but not on $\boldtheta$.\footnote {Because $A$ has full column rank, $Ax$ and $x$ can only vary proportionally when their directions are fixed.}
\end {enumerate}
\smallskip
Please refer to Figure \ref {fig:schematic} as needed. 
If $\col (A)$ is fixed, then $\boldkappa$ and $\vds$ depend only on the singular values of $A$, and $\boldtheta$ depends only on the orientation of $b$. Thus, $\boldkappa$ and $\boldtheta$ are separate sources of ill-conditioning for the solution. If $Ax$ has comparatively large components in singular vectors corresponding to the largest singular values of $A$, then $\vds \approx \boldkappa$ and the condition number $\chi_x (A)$ depends on $\boldkappasquared$ which was the discovery of \citet {GolubWilkinson1966}. Otherwise, $\boldkappasquared$ ``plays no role'' \citep [p.\ 251] {vanderSluis1975}. 

\begin {figure} 
\centering 
\includegraphics [scale=1] {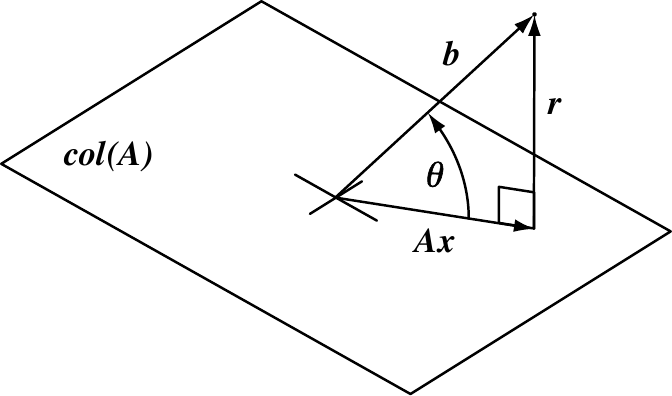}
\caption {Schematic of the least-squares problem, the projection $Ax$,  and the angle $\boldtheta$ between $Ax$ and $b$.}
\label {fig:schematic}
\end {figure}

\subsection {Conditioning example}
\label {sec:example}

This example illustrates the independent effects of $\boldkappa$, $\vds$, and $\boldtheta$ on $\chi_x (A)$. It is based on the example of \citet [p.\ 238] {Golub1996}.  Let
\begin {displaymath}
A = \left[ \begin {array} {c c} 1& 0\\ 0& \alpha\\ 0& 0\\ \end {array}
\right] ,
\quad 
b = \left[ \begin {array} {c} \beta \cos (\phi)\\ \beta \sin (\phi) \\ 1 \end {array} \right] \, ,
\quad 
\DA = \left[ \begin {array} {c c} 0& 0\\ 0& 0\\ 0& \epsilon \end {array}
\right] .
\end {displaymath}
where $0 < \epsilon \ll \alpha, \beta < 1$.
In this example, 
\begin {displaymath}
x = \left[ \begin {array} {c} \beta \cos (\phi) \\ {\beta \vphantom {(} \over \vphantom {(} \alpha} \sin (\phi) \\ \end {array} \right] ,
\quad 
r = \left[ \begin {array} {c} 0\\ 0\\ 1 \end {array} \right] ,
\quad 
x + \Dx  = \left[ \begin {array} {c} \beta \cos (\phi) \\[0.1ex] {\epsilon + \alpha \beta \sin (\phi) \vphantom {(} \over \vphantom {(}\alpha^2 + \epsilon^2} \\ \end {array} \right] \, .
\end {displaymath} 
The three terms in the condition number are
\begin {displaymath}
\boldkappa = {1 \over \alpha} \, , \qquad 
\vds = {1 \over \sqrt {[\alpha \cos (\phi)]^2 + [\sin (\phi)]^2}} \, , \qquad
\tan (\boldtheta) = {1 \over \beta} \, .
\end {displaymath}
These values can be independently manipulated by choosing $\alpha$, $\beta$, and $\phi$.  The tight upper bound for the condition number is
\begin {displaymath}
\chi_x (A) \le {1 \over \alpha} \left ( {1 \over \beta \sqrt {[\alpha \cos (\phi)]^2 + [\sin (\phi)]^2}} + 1 \right) \, .
\end {displaymath}
The relative change to the solution of the example
\begin {displaymath}
{\| \Dx \|_2 \over \| x \|_2} 
= {\epsilon \over \alpha \, \beta \,  \sqrt {[\alpha \cos (\phi)]^2 + [\sin (\phi)]^2}} + {\mathcal O} (\epsilon^2) \, .
\end {displaymath}
is close to the bound given by the condition number estimate and the relative change to $A$.

\section {Derivation of the condition numbers}
\label {sec:derivation}

\subsection {Notation}

The formula for the Jacobian matrix $J_x (b)$ of the solution $x = (A^t A)^{-1} b$ with respect to $b$ is clear.\footnote {The notation of section \ref {sec:separate} would introducing a name, $f_2$, for the function by which $x$ varies with $b$ when $A$ is held fixed, $x = f_2 (b)$, so that the notation for the Jacobian matrix is then $J_{f_2} (b)$. This pedantry will be discarded here to write $J_x(b)$ with $A$ held fixed; and similarly for $J_x(A)$ with $b$ held fixed.} For derivatives with respect to the entries of $A$, it is necessary to use the ``$\Vector$'' construction to order the matrix entries into a column vector; $\Vector (B)$ is the column of entries $B_{i,j}$ with $i,j$ in co-lexicographic order.\footnote {The alternative to placing the entries of matrices into column vectors is to use more general linear spaces and the Fr\'echet derivative. That approach seems unnecessarily abstract because the spaces have finite dimension.} The approximation is then
\begin {equation}
\label {eqn:total-differential}
\Dx = \JxA \, \Vector (\DA) + \Jxb \, \Db + \mbox {higher order terms}
\end {equation}
and upon taking norms
\begin {equation}
\label {eqn:differential-bound}
\| \Dx \| \le \underbrace {\| \JxA \|}_{\displaystyle \chi_x (A)} \, \| \DA \| + \underbrace {\displaystyle \| \Jxb \|}_{\displaystyle \chi_x (b)} \, \| \Db \| + o( \max \{ \| \DA \|, \, \| \Db \| \} ) \, ,
\end {equation}
where it is understood the norms on the two Jacobian matrices are induced from the following norms for $\DA$, $\Db$, and $\Dx$.

\subsection {Choice of Norms}
\label {sec:norms}

Equation (\ref {eqn:differential-bound}) applies for any choice of norms. In theoretical numerical analysis especially for least-squares problems the spectral norm is preferred. For $2$-norms the matrix condition number of $A^t A$ is the square of the matrix condition number of $A$. The norms used in this paper are thus,
\begin {equation}
\label {eqn:norms}
\| \Vector (\DA) \|_\scaleA  = {\| \DA \|_2 \over \scaleA} \, ,
\qquad
\| \Db \|_\scaleB = {\| \Db \|_2 \over \scaleB} \, ,
\qquad
\| \Dx \|_\scaleX = {\| \Dx \|_2 \over \scaleX} \, ,
\end {equation}
where $\scaleA$, $\scaleB$, $\scaleX$ are constant scale factors. These formulas define norms for $m \times n$ matrices, for $m$ vectors, and for $n$ vectors. The scaling makes the size of the changes relative to the problem of interest. The scaling used in equations (\ref {eqn:chixA}--\ref {eqn:specific-scaled-norms}) is
\begin {equation}
\label {eqn:scale-factors}
\scaleA = \| A \|_2 \, ,
\qquad
\scaleB = \| b \|_2 \, ,
\qquad
\scaleX = \| x \|_2 \, .
\end {equation}

\subsection {Condition number of {\itshape x\/} with respect to {\itshape b\/}} 
\label {sec:conditionwrtb}

From $x = (A^t A)^{-1} A^t b$ follows $\Jxb = (A^t A)^{-1} A^t$ and then for the scaling of equation (\ref {eqn:scale-factors})
\begin {equation}
\label {eqn:conditionwrtb}
\renewcommand {\arraycolsep} {0.125em}
\begin {array} {r c l}
\| \Jxb \|& 
=&
\displaystyle \max_{\Db} {\| \Jxb \, \Db \|_\scaleA \over \| \Db \|_\scaleB} 
= \displaystyle \max_{\Db} {\displaystyle \left( {\| (A^t A)^{-1} A^t \Db \|_2 \over \scaleX} \right) \over \displaystyle \left( {\| \Db \|_2 \over \scaleB} \right)}
= {\scaleB \over \scaleX \, \sigmamina}\\
&
=&
\displaystyle {\| b \|_2 \over \| A x \|_2} {\| A x \|_2 \over \| x \|_2 \, \sigmamina} = \sec (\boldtheta) \, \vds \, .
\end {array}
\end {equation}

\subsection {Condition number of {\itshape x\/} with respect to {\itshape A\/}} 
\label {sec:begin}

The Jacobian matrix $\JxA$ is most easily calculated from the total differential of the identity $F = A^t (b - A x) = 0$ with respect to $A$ and $x$, which is $J_F [\Vector (A)] \, \Vector (dA) + J_F (x) \, dx = 0$. Hence 
\begin {equation}
\label {eqn:hence}
dx = \underbrace {- [J_F (x)]^{-1} J_F [\Vector (A)]}_{\displaystyle J_x[\Vector (A)]} \Vector (dA)
\end {equation}
where $J_F (x) = - A^t A$ and where
\begin {equation}
\label {eqn:where}
J_F [\Vector (A)] 
=
\left[ 
\setlength {\arraycolsep} {0.25em}
\begin {array} {c c c}
r^t\\
& \ddots\\
&& \hspace {0.5em} r^t
\end {array}
\right]
- 
\left[ 
\setlength {\arraycolsep} {0.33em}
\begin {array} {c c c c c}
x_1 A^t& \cdots& x_n A^t
\end {array}
\right] ,
\end {equation}
in which $r = b - A x$ is the least-squares residual, and $x_i$ is the $i$-th entry of $x$. 

\subsection {Transpose formula for condition numbers}
\label {sec:transpose}

The desired condition number is the norm induced from the norms in equation (\ref {eqn:norms}). 
\begin {equation}
\label {eqn:induced}
\setlength {\arraycolsep} {0.25em}
\begin {array} {r c l}
\| \JxA \|
& =
& \displaystyle \max_{\DA} {\| \JxA \, \Vector (\DA) \|_\scaleX \over \| \Vector (\DA) \|_\scaleA}
\\ \noalign {\bigskip}
& =
& \displaystyle {\scaleA \over \scaleX} \max_{\DA} {\| \JxA \, \Vector (\DA) \|_2 \over \| \DA \|_2} 
\end {array}
\end {equation}
The numerator and denominator are vector and matrix $2$-norms, respectively. If $A$ is an $m \times n$ matrix, then the maximization in equation (\ref {eqn:induced}) has many degrees of freedom. An identity for operator norms can be applied to avoid this large optimization problem. 

Suppose $\reals^M$ and $\reals^N$ have the norms $\| \cdot \|_M$ and $\| \cdot \|_N$, respectively. If a problem with data $d \in \reals^N$ has a solution function $s = f(d) \in \reals^M$, then the condition number is the induced norm of the $M \times N$ Jacobian matrix,
\vspace {-1ex}
\begin {equation}
\label {eqn:condition}
\| J_f (d) \| = \max_{\Dd} {\| J_f(d) \, \Dd \|_M \over \| \Dd \|_N} \, .
\end {equation}
This optimization problem has $N$ degrees of freedom. An alternate expression is the norm for the transposed operator represented by the transposed matrix,\footnote {Equation (\ref {eqn:transpose}) is stated by 
\citet [chp.\ IV, p.\ 7, eqn.\ 4] {BourbakiTVS}, 
\citet [p.\ 478, lem.\ 2] {DunfordSchwartz1958},
\citet [p.\ 93, thm.\ 4.10, eqn.\ 2] {Rudin1973},
and 
\citet [p.\ 195, thm.\ 2$^\prime$, eqn.\ 3] {Yosida1974}.
The name of the transposed operator varies.  See \citet [chp.\ IV, p.\ 6, top] {BourbakiTVS} for ``transpose,''  Dunford and Schwartz (loc.\ cit.)\ and Rudin (loc.\ cit.)\ for ``adjoint,'' and \citet [p.\ 194, def.\ 1] {Yosida1974} for ``conjugate'' or ``dual.'' Some parts of mathematics use ``adjoint'' in the restricted context of Hilbert spaces, for example in linear algebra see \citet [pp.\ 168--174, sec.\ 5.1] {Lancaster1985}. That concept is actually a ``slightly different notion'' \citep [p.\ 479] {DunfordSchwartz1958} from the Banach space transpose used here.
}
\begin {equation}
\label {eqn:transpose}
\| J_f (d) \| = \| [J_f (d)]^t \|^* = \max_{\Ds} {\| [J_f (d)]^t \Ds \|_N^* \over \| \Ds \|_M^*} \, .
\end {equation}
The norm is induced from the dual norms $\| \cdot \|_M^*$ and $\| \cdot \|_N^*$ which must be determined. This optimization problem has $M$ degrees of freedom. Equation (\ref {eqn:transpose}) might be easier to evaluate, especially when the problem has many more data values than solution variables, $N \gg M$, as is often the case. 

Applying the formula (\ref {eqn:transpose}) for the norm of the transpose matrix to the equation (\ref {eqn:induced}) results in the simpler optimization problem,
\begin {equation}
\label {eqn:simpler}
\| \JxA \| =
{\scaleA \over \scaleX} \max_{\Dx} {\| \{ \JxA \}^t \Dx \|^*_2 \over \| \Dx \|^*_2} 
\end {equation}
The norm for the transposed Jacobian matrix is induced from the duals of the $2$-norms for matrices and vectors. The vector $2$-norm in the denominator is its own dual. So as not to interrupt the present discussion, some facts needed to evaluate the numerator are proved in section \ref {sec:advanced}: the dual of the matrix $2$-norm is the nuclear norm (section \ref {app:spectral}), and a formula for the nuclear norm is given (section \ref {app:rank2}).

\subsection {Condition number of {\itshape x\/} with respect to {\itshape A\/}, continued} 
\label {sec:conditionwrtAcontinued}

The application of equation (\ref {eqn:simpler}) requires evaluating the matrix-vector product in the numerator. Continuing the derivation of section \ref {sec:begin} from equation (\ref {eqn:where}), the vector-matrix product $v^t J_F [\Vector (A)]$ for some $v$ can be evaluated by straightforward multiplication, 
\begin {displaymath}
v^t J_F [\Vector (A)] 
=
\left[ 
\setlength {\arraycolsep} {0.25em}
\begin {array} {c c c}
v_1 r^t 
& \cdots
& v_n r^t 
\end {array}
\right]
- 
\left[ 
\setlength {\arraycolsep} {0.33em}
\begin {array} {c c c c c}
x_1 v^t A^t & \cdots& x_n v^t A^t 
\end {array}
\right] .
\end {displaymath}
This row vector, when transposed, is expressed more simply using $\Vector$ notation: the first part is $r$ scaled by each entry of $v$, $\Vector (r v^t)$, the second part is $A v$ scaled by each entry of $x$, $\Vector (A v x^t) $. Altogether
\begin {displaymath}
J_F [\Vector (A)]^t v
=
\Vector (rv^t - Avx^t) \, .
\end {displaymath}
Substituting $v = \{- [J_F (x)]^{-1}\}^t \Dx = (A^t A)^{-1} \Dx$ for some $\Dx$ gives, by equation (\ref {eqn:hence}),
\begin {displaymath}
\JxA^t \Dx 
= 
\Vector \left\{ r \, [(A^t A)^{-1} \Dx]^t - A \, [(A^t A)^{-1} \Dx] \, x^t \right\} \, ,
\end {displaymath}
or equivalently,
\begin {equation}
\label {eqn:JxADx}
\Matrix \{ \JxA^t \Dx \} 
=
u_1^{} v_1^t + u_2^{} v_2^t 
\end {equation}
where ``$\Matrix$'' is the inverse of ``$\Vector$,'' and
\begin {equation}
\label {eqn:vectors}
\setlength {\arraycolsep} {0.125em}
\begin {array} {r c l r c l}
u_1& =& r = b - Ax,& 
v_1& =& (A^t A)^{-1} \Dx,\\ \noalign {\smallskip}
u_2& =& - A (A^t A)^{-1} \Dx, \qquad&
v_2& =& x .
\end {array}
\end {equation}
The matrix on the right side of equation (\ref {eqn:JxADx}) has rank $2$. Moreover, the two rank $1$ pieces are mutually orthogonal because the least-squares residual $r$ is orthogonal to the coefficient matrix $A$. With these replacements equation (\ref {eqn:simpler}) becomes
\begin {equation}
\label {eqn:inducedtransposed}
\big \| \JxA \big\|
=
{\scaleA \over \scaleX} \max_{\| \Dx \|_2 = 1} \| u_1^{} v_1^t + u_2^{} v_2^t \|^*_2 \, .
\end {equation}

Lemma \ref {lem:kahan} shows that the dual of the spectral matrix norm is the matrix norm that sums the singular values of the matrix, which is called the nuclear  norm. Lemma \ref {lem:rank2} then evaluates this norm for rank $2$ matrices to find that the objective function of equation (\ref {eqn:inducedtransposed}) is
\begin {equation}
\label {eqn:numerator}
\sqrt {\strut \| u_1 \|_2^2 \, \| v_1 \|_2^2 + \| u_2 \|_2^2 \, \| v_2 \|_2^2 + 2 \, \| u_1 \|_2 \, \| v_1 \|_2 \, \| u_2 \|_2 \, \| v_2 \|_2 \, \cos (\theta_u - \theta_v)} \, ,
\end {equation}
where $\theta_u$ is the angle between $u_1$ and $u_2$, and $\theta_v$ is the angle between $v_1$ and $v_2$, and both angles should be taken from $0$ to $\pi$. Since $u_1$ is orthogonal to $u_2$ therefore $\theta_u = \pi / 2$ and then $| \theta_u - \theta_v | \le \pi / 2$ so $\cos (\theta_u - \theta_v)$ is not negative. This means the maximum lies between the lower and upper limits
\begin {equation}
\label {eqn:lowerandupper}
\sqrt {\strut \| u_1 \|_2^2 \, \| v_1 \|_2^2 + \| u_2 \|_2^2 \, \| v_2 \|_2^2} \quad \mbox {and} \quad \| u_1 \|_2 \, \| v_1 \|_2 + \| u_2 \|_2 \, \| v_2 \|_2 \, .
\end {equation}
With $\| \Dx \|_2$ restricted to $1$, the lower bound and also the upper bound attain their maxima when $\Dx$ is a right singular unit vector for the smallest singular value of $A$,
\begin {equation}
\label {eqn:maxima}
\| u_1 \|_2 \, \| v_1 \|_2 = {\| r \|_2 \over (\sigmamina)^2} \quad \mbox {and} \quad \| u_2 \|_2 \, \| v_2 \|_2 = {\| x \|_2 \over \sigmamina} \, .
\end {equation}
Some value of $\| u_1^{} v_1^t + u_2^{} v_2^t \|^*_2$ lies between the limits when $\Dx$ is a right singular unit vector for the smallest singular value of $A$. Because these are the largest possible limits, the maximum value must lie between them as well. These limits must be multiplied by the coefficient $\scaleA / \scaleX$ in equation (\ref {eqn:inducedtransposed}) to obtain bounds for the norm of the Jacobian matrix. 

\subsection {Summary of condition numbers} 
\label {sec:conditionwrtAfinished}

\ \par

\begin {theorem} 
[\scshape Spectral condition numbers]
\label {thm:condition-numbers}
For the full column rank linear least-squares problem with solution $x = (A^tA)^{-1} A^t b$, and for the scaled norms of equation (\ref {eqn:norms}) with scale factors $\scaleA$, $\scaleB$, and $\scaleX$,
\begin {displaymath}
\chi_x (b) =  {\scaleB \over \scaleX \sigmamin}
\qquad
\chi_x (A) = {\scaleA \over \scaleX} \max_{\| \Delta x \|_2 = 1} \sigma_1 + \sigma_2 \, ,
\end {displaymath}
where $\sigma_1$ and $\sigma_2$ are the singular values of the rank $2$ matrix $u_1^{} v_1^t + u_2^{} v_2^t$ \,for 
\begin {displaymath}
\setlength {\arraycolsep} {0.125em}
\begin {array} {r c l r c l}
u_1& =& r = b - Ax,& 
v_1& =& (A^t A)^{-1} \Dx,\\ \noalign {\smallskip}
u_2& =& - A (A^t A)^{-1} \Dx, \qquad&
v_2& =& x .
\end {array}
\end {displaymath}
The value of $\chi_x (A)$ lies between the lower limit of Malyshev and the upper limit of Bj\"orck,
\begin {displaymath}
{\scaleA \over \scaleX \sigmamin} \sqrt { \left( {\| r \|_2 \over \sigmamin} \right)^2 + \| x \|_2^2}
\, \le \,
\chi_x (A)
\, \le \,
{\scaleA \over \scaleX \sigmamin} \left( {\| r \|_2 \over \sigmamin} + \| x \|_2 \right) \, .
\end {displaymath}
The upper bound exceeds $\chi_x (A)$ by at most a factor $\sqrt 2$. The formula for $\chi_x(b)$ and the limits for $\chi_x(A)$ simplify to equations (\ref {eqn:chixA}, \ref {eqn:chixb}) for the scale factors in equation (\ref {eqn:scale-factors}). 
\end {theorem}

\begin {proof} 
Section \ref {sec:conditionwrtb} derives $\chi_x(b)$, and sections \ref {sec:begin}--\ref {sec:conditionwrtAfinished} derive $\chi_x(A)$ and the bounds. 
\end {proof}

\section {Discussion}
\label {sec:comparison}

\subsection {Example of strict limits}

The condition number $\chi_x (A)$ in theorem \ref {thm:condition-numbers} can lie strictly between the limits of Bj\"orck and Malyshev. For the example of section \ref {sec:example}, the rank $2$ matrix in the theorem is
\begin {displaymath}
\renewcommand {\arraystretch} {1.25}
u_1^{} v_1^t + u_2^{} v_2^t = \left[ \begin {array} {c c} 
- \beta \cos (\phi) \Delta x_1& - \beta \sin (\phi) \Delta x_1 / \alpha\\ 
- \beta \cos (\phi) \Delta x_2 / \alpha& - \beta \sin (\phi) \Delta x_2 / \alpha^2\\ 
\Delta x_1& \Delta x_2 / \alpha^2 \\ \end {array} \right] .
\end {displaymath}
For the specific values $\alpha = 1/10$, $\beta = 1$, $\phi = \pi / 10$, the sum of the singular values of this matrix can be numerically maximized over $\| \Delta x \|_2 = 1$ to evaluate the condition number with the following results.
\begin {displaymath}
\renewcommand {\arraystretch} {1.25}
\renewcommand {\tabcolsep} {0.25em}
\begin {tabular} {r c l l}
$\displaystyle \boldkappa (\vds \tan (\boldtheta) + 1)$& $=$& 40.928\ldots& upper limit of Bj\"orck\\
$\displaystyle \chi_x (A)$& $=$& 35.193\ldots& condition number\\
$\displaystyle \boldkappa \, \big( \strut [\vds  \tan (\boldtheta)]^2 + 1\big)^{1/2}$& $=$& 32.505\ldots& lower limit of Malyshev\\
\end {tabular}
\end {displaymath}
These calculations were done with Mathematica \citep {Wolfram2003}. 

\subsection {Exact formulas for some condition numbers}
\label {sec:exact}

\begin {table} 
\caption {\it Cases for which condition numbers have been determined for the full column rank least-squares problem, $\min_x \| b - A x \|_2$. All the formulas are for $\chi_x(A)$ except Grattan's formula is for $\chi_x(A,b)$. Notation: $r$ is the residual, $\sigmamin$ is the smallest singular value of $A$. In column 5, ``approx'' means the value in column 4 approximates the condition number, ``exact'' means it is the condition number for the chosen norms.}
\label {tab:LLS}
\newcommand {\textstrut} {\vrule depth1.125ex height2.375ex width\strutwidth}
\newcommand {\leftbox} [3] {\begin {minipage}{#1} \vrule depth0ex height4.0ex width\strutwidth #2\\ \scriptsize #3\vrule depth3.0ex height0ex width\strutwidth \end {minipage}}
\setlength {\arraycolsep} {0.6em}
\vspace {-4ex}
\begin {displaymath}
\small
\begin {array} {| l | c | c | c | c |}
\cline {2-3}
\multicolumn {1} { c |} {\textstrut}
& \multicolumn {2} { c |} {\mbox {norms}}
\\
\hline 
\multicolumn {1} {| c |} {\textstrut \mbox {source}}
& \mbox {data}
& \mbox {solution}
& \mbox {formula}
& \mbox {status}
\\ \hline \hline
\leftbox {10em} 
{Bj\"orck and theorem \ref {thm:condition-numbers}} 
{\nocite {Bjorck1996}(1996, p.\ 31, eqn.\ 1.4.28)}
& \displaystyle {\| \DA \|_2 \over \| A \|_2}
& \displaystyle {\| \Dx \|_2 \over \| x \|_2}
& \displaystyle {\| A \|_2 \over \strut \sigmamin} \left( {\| r \|_2 \over \strut \sigmamin \, \| x \|_2} \, + 1 \right)
& \mbox {approx}
\\
\leftbox {10em} 
{Geurts} 
{\nocite {Geurts1982}(1982,  p.\ 93, eqn.\ 4.3)}
& \displaystyle {\| \DA \|\sub{F} \over \| A \|\sub{F}}
& \displaystyle {\| \Dx \|_2 \over \| x \|_2}
& \displaystyle {\| A \|\sub{F} \over \sigmamin^{}} \sqrt{ {\| r \|_2^2 \over \sigmamin^2 \, \| x \|_2^2} + 1}
& \mbox {exact}
\\
\multicolumn {1} {| l} {\leftbox {10em} {Gratton} {\nocite {Gratton1996}(1996, p.\ 525, eqn.\ 2.1)}}
& \hspace {-4em} \displaystyle \left\| \strut [ \alpha \, \DA, \beta \, \Db \, ] \right \|\sub{F}
& \| \Dx \|_2
& \displaystyle {1 \over \sigmamin^{}} \sqrt{ {\| r \|_2^2 \over \alpha^2 \sigmamin^2} + {\| x \|_2^2 \over \alpha^2} + {1 \over \beta^2}}
& \mbox {exact}
\\
\leftbox {10em} 
{Malyshev} 
{\nocite {Malyshev2003}(2003, p.\ 1187, eqn.\ 1.8)}
& \displaystyle {\| \DA \|\sub{F} \over \| A \|_2}
& \displaystyle {\| \Dx \|_2 \over \| x \|_2}
&& \mbox {exact}
\\
\leftbox {11.5em} 
{Malyshev and theorem \ref {thm:condition-numbers}} 
{\nocite {Malyshev2003}(2003, p.\ 1189, eqn.\ 2.4 and line --6)}
& \displaystyle {\| \DA \|_2 \over \| A \|_2}
& \displaystyle {\| \Dx \|_2 \over \| x \|_2}
& \mbox{\hspace {-1.5em} \raisebox{4.0ex}[0pt][0pt]{$\displaystyle \left.
\vrule depth7ex height0ex width 0pt \right\} \hspace {0.5em} {\| A \|_2 \over \sigmamin^{}} \sqrt{ {\| r \|_2^2 \over \sigmamin^2 \, \| x \|_2^2} +
1}$}}
& \mbox {approx}
\\ \hline
\end {array}
\end {displaymath}
\vspace*{-2ex}
\end {table}

Table \ref {tab:LLS} lists several condition numbers or approximations to condition numbers for least-squares solutions. The three exact values measure changes to $A$ by the Frobenius norm, while the two approximate values are for the spectral norm. The difference can be attributed to the ease or difficulty of solving the maximization problem in equation (\ref {eqn:inducedtransposed}). The dual spectral norm of the rank $2$ matrix involves a trigonometric function, $\cos (\pi / 2 - \theta_v)$ in equation (\ref {eqn:numerator}), whose value only can be estimated. If a Frobenius norm were used instead, then lemma \ref {lem:rank2} shows the dual norm of the rank $2$ matrix involves an expression, $\cos (\pi / 2) \cos (\theta_v)$, whose value is zero, which greatly simplifies the maximization problem.

\subsection {Overestimates of condition numbers}
\label {sec:overestimate}

Many error bounds in the literature combine $\chi_x(A) + \chi_x(b)$ in the manner of equation (\ref {eqn:double}),
\begin {eqnarray}
\nonumber 
{\| \Dx \|_2 \over \| x \|_2}
& \le
& \chi_x (A,b) \, \epsilon + o (\epsilon) \qquad \mbox {attainable}
\\ 
\label {eqn:bestbound-1}
& \le
& \big[ \chi_x(A) + \chi_x(b) \big] \epsilon + o (\epsilon) \qquad \mbox {overestimate by at most $\times 2$}
\\ \noalign {\smallskip}
\label {eqn:bestbound-2}
& \le
& \big[ \boldkappa (\vds \tan (\boldtheta) + 1) + \vds \sec (\boldtheta) \big] \epsilon + o (\epsilon) \quad \mbox {further at most $\times \sqrt 2$}
\\ \noalign {\smallskip} 
\label {eqn:bestbound-3}
& = 
& \left( {\| A \|_2 \| r \|_2 \over \sigmamin^2 \| x \|_2} + {\| A \|_2 \over \sigmamin} + {\| b \|_2 \over \sigmamin \| x \|_2} \right) \epsilon +o (\epsilon) \, ,
\end {eqnarray}
where $\epsilon = \max \{ \| \DA \|_2 / \| A \|_2, \| \Db \|_2 / \| b \|_2 \}$. Bounds (\ref {eqn:bestbound-1}, \ref {eqn:bestbound-2}) are larger than the attainable bound by at most factors $2$ and $2 \sqrt 2$, respectively, by equation (\ref {eqn:sum-6}) and theorem \ref {thm:condition-numbers}.

Some bounds are yet larger. \citet [p.\ 382, eqn.\ 20.1] {Higham2002} reports
\begin {eqnarray*}
{\| \Dx \|_2 \over \| x \|_2}
& \le
& \boldkappa \epsilon \left( 2 + ( \boldkappa + 1) {\| r \|_2 \over \| A \|_2 \| x \|_2} \right) + {\mathcal O} (\epsilon^2)
\\ \noalign {\medskip}
& = 
& \left( {\| A \|_2 \| r \|_2 \over \sigmamin^2 \| x \|_2} + {\| A \|_2 \over \sigmamin} + {\| A \|_2 \| x \|_2 + \| r \|_2 \over \sigmamin \| x \|_2} \right) \epsilon + {\mathcal O} (\epsilon^2) \, .
\end {eqnarray*}
This bound is an overestimate in comparison to equation (\ref {eqn:bestbound-3}).

An egregious overestimate occurs in an error bound that appears to have originated in the 1983 edition of the popular textbook of \citet [p.\ 242, eqn.\ 5.3.8] {Golub1996}. The overestimate is restated by \citet [p.\ 50] {Anderson1992} in the LAPACK documentation, and by \citet [p.\ 117] {Demmel1997},
\begin {equation}
\label {eqn:GVL}
{\| \Dx \|_2 \over \| x \|_2} 
\le 
\left[ 2 \sec (\boldtheta) \, \boldkappa + \tan (\boldtheta) \, \boldkappasquared \right] \epsilon + {\mathcal O} (\epsilon^2) \, .
\end {equation}
In comparison with equation (\ref {eqn:bestbound-2}) this bound replaces $\vds$ by $\boldkappa$ and replaces $1$ by $\sec (\boldtheta)$. An overestimate by a factor of $\boldkappa$ occurs for the example of section \ref {sec:example} with $\alpha \ll 1$, $\beta = 1$, and $\phi = {\pi \over 2}$. In this case the ratio of equation (\ref {eqn:GVL}) to equation (\ref {eqn:bestbound-2}) is
\begin {displaymath}
{2 \sec (\boldtheta) \, \boldkappa + \tan (\boldtheta) \, \boldkappasquared \over \boldkappa (\vds \tan (\boldtheta) + 1) + \vds \sec (\boldtheta)} = {\displaystyle {1 + 2 \alpha \sqrt {1 + \beta^2} \over \alpha^2 \beta} \over \displaystyle {1 + \beta + \alpha \sqrt {1 + \beta^2} \over \alpha \beta}} \approx {1 \over 2 \, \alpha} = {\boldkappa \over 2} \, .
\end {displaymath}

\section {Norms of operators on normed linear spaces of finite dimension}
\label {sec:advanced}

\subsection {Introduction}

This section describes the dual norms in the formulas of sections \ref {sec:transpose} and \ref {sec:conditionwrtAcontinued}. The actual mathematical concept is a norm for the dual space. However, linear algebra ``identifies'' a space with its dual, so the concept becomes a ``dual norm'' for the same space. This point of view is appropriate for Hilbert spaces, but it omits an important level of abstraction. As a result, the linear algebra literature lacks a complete development of finite dimensional normed linear (Banach) spaces. Rather than make functional analysis a prerequisite for this paper, here the identification approach is generalized to give dual norms for spaces other than column vectors (which is needed for data in matrix form), but only as far as the dual norm itself in section \ref {app:duals}. Section \ref {app:spectral} gives the formula for the dual of the spectral matrix norm. Section \ref {app:rank2} evaluates the norm for matrices of rank $2$. 

{\it Banach spaces are needed in this paper because the norms used in numerical analysis are not necessarily those of a Hilbert space.} The space of $m \times n$ matrices viewed as column vectors has been given the spectral matrix norm in equation (\ref {eqn:norms}). If the norm were to make the space a Hilbert space, then the norm would be given by an inner product. There would be an $mn \times mn$ symmetric matrix, $S$, so that for every $m \times n$ matrix $B$,
\begin {displaymath}
\| B \|_2
= 
\sqrt {\strut [\Vector (B)]^t \; S \; \Vector (B)} \, ,
\end {displaymath}
which is impossible. 

\subsection {Duals of normed spaces}
\label {app:duals} 

If $\X$ is a finite dimensional vector space over $\reals$, then the dual space $\X^*$ consists of all linear transformations $f : \X \rightarrow \reals$, called functionals. If $\X$ has a norm, then $\X^*$ has the usual operator norm given by
\begin {equation}
\label {eqn:operatorNorm}
\| f \| = \sup_{\| x \| = 1} f (x) \, .
\end {equation}
One notation is used for both norms because whether a norm is for $\X$ or $\X^*$ can be decided by what is inside. 

For a finite dimensional $\X$ with a basis $e_1$, $e_2$, $\ldots\,$, $e_n$, the dual space has a basis $g_1$, $g_2$, $\ldots\,$, $f_n$ defined by $f_i (e_j) = \delta_{i,j}$ where $\delta_{i,j}$ is Kronecker's delta function. In linear algebra for finite dimensional spaces, it is customary to represent the arithmetic of $\X^*$ in terms of $\X$ under the transformation $T : \X \rightarrow \X^*$ defined on the bases by $T(e_i) = f_i$. \textit {This transformation is not unique because it depends on the choices of bases.\/} Usually $\X$ has a favored or ``canonical'' basis whose $T$ is said to ``identify'' $\X^*$ with $\X$. Under this identification the norm for the dual space then is regarded as a norm for the original space.

\begin {definition}
[\scshape Dual norm]
Let $\X = \reals^m$ have norm $\| \cdot \|$ and let $T$ identify $\X$ with the dual space $\X^*$. The dual norm for $\X$ is 
\begin {displaymath}
\| v \|^* = \| T(v) \| 
\end {displaymath}
where the right side is the norm in equation (\ref {eqn:operatorNorm}) for the dual space. 
\end {definition}

The notation $\| \cdot \|^*$ avoids confusing the two norms for $\X$. There seems to be no standard notation for the dual norm; others are $\| \cdot \|_D$, $\| \cdot \|^D$, and $\| \cdot \|_{\rm d}$ which are used respectively by \citet [p.\ 107, eqn.\ 6.2] {Higham2002}, \citet [p.\ 275, def.\ 5.4.12] {Horn1985}, and \citet [p.\ 381, eqn.\ 1] {Lancaster1985}.

\subsection {Dual of the spectral matrix norm}
\label {app:spectral}

The space $\reals^{m \times n}$ of real $m \times n$ matrices has a canonical basis consisting of the matrices $E^{(i,j)}$ whose entries are zero except the $i,j$ entry which is $1$. This basis identifies a matrix $A$ with the functional whose value at a matrix $B$ is $\sum_{i,j} A_{i,j} B_{i,j} = \tr (A^t B)$.

\begin {lemma} 
[\scshape Dual of the spectral matrix norm]
\label {lem:kahan}
The dual norm of the spectral matrix norm with respect to the aforementioned canonical basis for $\reals^{m \times n}$ is given by $\| A \|_2^* = \| \sigma(A) \|_1$, where $\sigma (A) \in \reals^{\min \{m, n\}}$ is the vector of $A$'s singular values including multiplicities. That is, $\| A \|_2^*$ is the sum of the singular values of $A$  with multiplicities, which is called the nuclear norm or the trace norm.
\end{lemma}

\begin {proof} 
{\it (Supplied by \citet {Kahan2003}.)} Let $A = U \Sigma V^t$ be a ``full'' singular value decomposition of $A$, where both $U$ and $V$ are orthogonal matrices, and where $\Sigma$ is an $m \times n$ ``diagonal'' matrix whose diagonal entries are those of $\sigma (A)$. The trace of a square matrix, $M$, is invariant under conjugation, $V^{-1} M V$, so
\begin {displaymath}
\| A \|_2^*
=
\sup_{\| B \|_2 = 1} \tr (A^t B)
=
\sup_{\| B \|_2 = 1} \tr (V \Sigma^t \, U^t B)
=
\sup_{\| B \|_2 = 1} \tr (\Sigma^t \, U^t B \, V) \, .
\end {displaymath}
Since $\| U^t B V \|_2 = \| B \|_2 = 1$, the entries of $U^t B V$ are at most $1$ in magnitude, and therefore $| \tr (\Sigma^t \, U^t B \, V) | \le \tr (\Sigma^t)$. This upper bound is attained for $B = U D V^t$ where $D$ is the $m \times n$ ``identity'' matrix. \end {proof}

An alternate proof is offered by the work of \nocite {Taub1963} von \citet {vonNeumann1937siam}. He studied a special class of norms for $\reals^{m \times n}$. A symmetric gauge function of order $p$ is a norm for $\reals^p$ that is unchanged by every permutation and sign change of the entries of the vectors. Such a function applied to the singular values of a matrix always defines a norm on $\reals^{m \times n}$. For example, $\| A \|_2 = \| \sigma (A) \|_\infty$ where as in lemma \ref {lem:kahan} $\sigma (A)$ is the length $\min \{ m, n \}$ column vector of singular values for $A$. The dual of this norm is given by the dual norm for the singular values vector, see \citet [p.\ 78, lem.\ 3.5] {Stewart1990}. 

\begin {proof} 
{\it (In the manner of John von Neumann.)} By the aforementioned lemma to von Neumann's gauge theorem, $\| A \|_2^* = \| \sigma (A) \|_\infty^* = \| \sigma (A) \|_1$.
\end {proof}

\subsection {Rank 2 Matrices}
\label {app:rank2}

This section finds singular values of rank $2$ matrices to establish some norms of the matrices that simplify the condition numbers in equation (\ref {eqn:inducedtransposed}).

\begin {lemma}
[\scshape Frobenius and nuclear norms of rank 2 matrices.]
\label {lem:rank2}
If $u_1, u_2 \in \reals^m$ and $v_1, v_2 \in \reals^n$, then (Frobenius norm)
\begin {displaymath}
\begin {array} {l}
\left\| u_1^{} v_1^t + u_2^{} v_2^t \right\|\sub{F}^* =
\left\| u_1^{} v_1^t + u_2^{} v_2^t \right\|\sub{F} =
\\
\noalign {\medskip} 
\hspace {2em} \sqrt {\strut \| u_1 \|_2^2 \, \| v_1 \|_2^2 + \| u_2 \|_2^2
\, \| v_2 \|_2^2 + 2 \, \| u_1 \|_2 \, \| v_1 \|_2 \, \| u_2 \|_2 \, \| v_2
\|_2 \, \cos (\theta_u) \cos (\theta_v)} \, ,
\end {array}
\end {displaymath}
and (nuclear norm, or trace norm)
\begin {displaymath}
\begin {array} {l}
\left\| u_1^{} v_1^t + u_2^{} v_2^t \right\|_2^* =
\\
\noalign {\medskip} 
\hspace {2em} \sqrt {\strut \| u_1 \|_2^2 \, \| v_1 \|_2^2 + \| u_2 \|_2^2
\, \| v_2 \|_2^2 + 2 \, \| u_1 \|_2 \, \| v_1 \|_2 \, \| u_2 \|_2 \, \| v_2
\|_2 \, \cos (\theta_u - \theta_v)} \, ,
\end {array}
\end {displaymath}
where $\theta_u$ is the angle between $u_1$ and $u_2$, and $\theta_v$ is
the angle between $v_1$ and $v_2$. Both angles should be taken from $0$ to $\pi$.
\end{lemma}

\begin {proof}
If any of the vectors vanish, then the formulas are clearly true, so it may be assumed that the vectors are nonzero. The strategy of the proof is to represent the rank $2$ matrix as a $2 \times 2$ matrix whose singular values can be calculated. Since singular values are wanted, it is necessary that the bases for the $2 \times 2$ representation be orthonormal. 

To that end, let $w_1$ and $w_2$ be orthogonal unit vectors with $u_1 = \alpha_1 w_1$ and $u_2 = \alpha_2 w_1 + \beta w_2$. The coefficients' signs are indeterminate, so without loss of generality assume $\alpha_1 \ge 0$ and $\beta \ge 0$, in which case
\begin {displaymath}
\alpha_1 = \| u_1 \|_2
\qquad
\alpha_2 = {u_1^t u_2^{} \over \| u_1 \|_2}
\qquad
\beta = \left\| u_2 - \left( {u_1^t u_2^{} \over \| u_1 \|_2} \right)
\left( {u_1 \over \| u_1 \|_2} \right) \, \right\|_2 \, .
\end {displaymath} Similarly, let $x_1$ and $x_2$ be mutually orthogonal
unit vectors with $v_1 = \gamma_1 x_1$ and $v_2 = \gamma_2 x_1 + \delta
x_2$. Again without loss of generality $\gamma_1 \ge 0$ and $\delta \ge 0$
so that
\begin {displaymath}
\gamma_1 = \| v_1 \|_2
\qquad
\gamma_2 = {v_1^t v_2^{} \over \| v_1 \|_2}
\qquad
\delta = \left\| v_2 - \left( {v_1^t v_2^{} \over \| v_1 \|_2} \right)
\left( {v_1 \over \| v_1 \|_2} \right) \, \right\|_2 \, .
\end {displaymath}
Notice that
\begin {displaymath}
\beta^2 = \| u_2 \|_2^2 - \left( {u_1^t u_2^{} \over \| u_1 \|_2} \right)^2
\qquad
\delta^2 = \| v_2 \|_2^2 - \left( {v_1^t v_2^{} \over \| v_1 \|_2}
\right)^2 \, .
\end {displaymath}
Let $G = u_1^{} v_1^t + u_2^{} v_2^t$. A straightforward calculation shows that, with respect to the orthonormal basis consisting of $x_1$ and $x_2$, the matrix $G^t G$ is represented by the matrix 
\begin {displaymath}
M = \left[ \begin {array} {c c} {\beta }^2\,{{{\gamma }_2}}^2 + 
{\left( {{\alpha }_1}\,{{\gamma }_1} + 
{{\alpha }_2}\,{{\gamma }_2} \right) 
}^2& {\beta }^2\,\delta \,{{\gamma }_2} + 
\delta \,{{\alpha }_2}\,
\left( {{\alpha }_1}\,{{\gamma }_1} + 
{{\alpha }_2}\,{{\gamma }_2} \right)\\ \noalign {\medskip} {\beta
}^2\,\delta \,{{\gamma }_2} + 
\delta \,{{\alpha }_2}\,
\left( {{\alpha }_1}\,{{\gamma }_1} + 
{{\alpha }_2}\,{{\gamma }_2} \right)& {\beta }^2\,{\delta }^2 + 
{\delta }^2\,{{{\alpha }_2}}^2 \end {array} \right] \, .
\end {displaymath}
The desired norms are now given in terms of the eigenvalues of $M$, $\lambda_\pm$,
\begin {displaymath}
\| G \|\sub{F}
=
\sqrt {\lambda_+ + \lambda_-} = \sqrt {\tr (M)}
\quad \mbox {and} \quad
\| G \|_2^*
=
\sqrt {\lambda_+} + \sqrt {\lambda_-} \, .
\end {displaymath}

The expression for $\| G \|_2^*$ requires further analysis. For any $2 \times 2$ matrix $M$,
\begin {displaymath}
\lambda_\pm = {\tr (M) \over 2} \pm \sqrt { \left( {\tr (M) \over 2}
\right)^2 - \det (M)} \, .
\end {displaymath}
In the present case these eigenvalues are real because the $M$ of interest is symmetric, and $\det (M) \ge 0$ because it is also positive semidefinite. Altogether $[\tr (M)]^2 \ge 4 \det (M) \ge 0$, so $\tr (M) \ge 2 \det (M) \ge 0$. These bounds prove the following quantities are real, and it can be verified they are the square roots of the eigenvalues of $M$,
\begin {displaymath}
\sqrt {\lambda_\pm} = \sqrt {{\tr (M) \over 4} + \sqrt {\det (M) \over 4}}
\pm \sqrt {{\tr (M) \over 4} - \sqrt {\det (M) \over 4}} \, ,
\end {displaymath}
thus
\begin {displaymath}
\| G \|_2^*
=
\sqrt {\lambda_+} + \sqrt {\lambda_-}
=
\sqrt {\tr (M) + 2 \sqrt {\det (M)}} \, .
\end {displaymath}

In summary, the desired quantities $\| G \|_2^*$ and $\| G \|\sub{F}^{}$ have been expressed in terms of $\det (M)$ and $\tr (M)$ which the expression for $M$ expands into formulas of $\alpha_i$, $\beta$, $\gamma_i$, and $\delta$. These in turn expand to expressions of $u_i$ and $v_i$. It is remarkable that the ultimate expressions in terms of $u_i$ and $v_i$ are straightforward,
\begin {eqnarray*}
\tr (M)
& =
& \| u_1 \|_2^2 \, \| v_1 \|_2^2 + \| u_2 \|_2^2 \, \| v_2 \|_2^2 + 2 \,
(u_1^t u_2^{}) (v_1^t v_2^{})\\
\noalign {\medskip}
& =
& \| u_1 \|_2^2 \, \| v_1 \|_2^2 + \| u_2 \|_2^2 \, \| v_2 \|_2^2 + 2 \, \|
u_1 \|_2 \, \| v_1 \|_2 \, \| u_2 \|_2 \, \| v_2 \|_2 \, \cos (\theta_u)
\cos (\theta_v)\\
\noalign {\bigskip}
\det (M)
& =
& \left( \| u_1 \|_2^2 \, \| u_2 \|_2^2 - (u_1^t u_2^{})^2 \right) \left(
\| v_1 \|_2^2 \, \| v_2 \|_2^2 - (v_1^t v_2^{})^2 \right)\\
\noalign {\medskip}
& =
& \left( \| u_1 \|_2^2 \, \| u_2 \|_2^2 (\sin (\theta_u))^2 \right) \left(
\| v_1 \|_2^2 \, \| v_2 \|_2^2 (\sin (\theta_v))^2 \right) \, ,
\end {eqnarray*}
where $\theta_u$ is the angle between $u_1$ and $u_2$, and similarly for $\theta_v$. The formula for $\| G \|\sub{F}$ is established. The formula for $\| G \|_2^*$ simplifies, using the difference formula for cosine, to the one in the statement of the lemma. Since the positive root of $\sqrt {\det (M)}$ is wanted, the angles should be chosen from $0$ to $\pi$ so the squares of the sines can be removed without inserting a change of sign. These calculations have been verified with Mathematica \citep {Wolfram2003}. 
\end {proof}

\raggedright
\bibliographystyle {plainnat}
\bibliographystyle {siam}

\end{document}